\documentclass[12pt,english]{amsart}
                                                        
\usepackage[inner=1in,outer=1in,top=1in,bottom=1in, marginparwidth=.7in,marginparsep=.1in]{geometry}

\usepackage[T1]{fontenc}
\usepackage{amssymb}

\usepackage{colordvi}
\usepackage{graphicx}

\usepackage{paralist}
\usepackage{babel}
\usepackage{amsmath}
\usepackage{amsthm}
\usepackage{tikz}
\usepackage{todonotes}

\newcommand\C{\mathbb{C}}

\renewcommand\P{\mathbb{P}}

\newcommand\R{\mathbb{R}}

\newcommand\ii{{\rm i}}

\newcommand\B{\mathcal{B}}
\newcommand\sA{\mathcal{A}}
\newcommand\cI{\mathcal{I}}
\newcommand\cJ{\mathcal{J}}

\DeclareMathOperator\rank{rank}

\DeclareMathOperator\sym{sym}
\DeclareMathOperator\Skew{skew}
\DeclareMathOperator\Herm{Herm}
\DeclareMathOperator\Res{Res}

\theoremstyle{plain}
\newtheorem{Thm}{Theorem}[section]
\newtheorem{Prop}[Thm]{Proposition}

\newtheorem{Lemma}[Thm]{Lemma}
\newtheorem{Conjecture}[Thm]{Conjecture}

\newtheorem*{Thm*}{Theorem}
\newtheorem*{Prop*}{Proposition}
\newtheorem*{Cor*}{Corollary}
\newtheorem*{Lemma*}{Lemma}
\newtheorem*{Conjecture*}{Conjecture}
\newtheorem*{Conjecture4M-4}{The $\mathbf{4M-4}$ Conjecture \cite{BCMN:13}}

\theoremstyle{definition}

\newtheorem{Def}[Thm]{Definition}
\newtheorem{Example}[Thm]{Example}
\newtheorem{Remark}[Thm]{Remark}

\newtheorem*{Constr*}{Construction}
\newtheorem*{Def*}{Definition}
\newtheorem*{Example*}{Example}
\newtheorem*{Remark*}{Remark}

\newcommand{\Omit}[1]{\textcolor{red}{#1}}
\renewcommand{\Omit}[1]{}

\begin{document}
\title[An algebraic characterization of injectivity in phase retrieval]{An algebraic characterization of injectivity\\ in phase retrieval}
\author{Aldo Conca}
\author{Dan Edidin}
\author{Milena Hering}
\author{Cynthia Vinzant}

\date{}

\begin{abstract} 
A complex frame is a collection of vectors that span $\mathbb{C}^M$ and define measurements, 
called intensity measurements, on vectors in $\mathbb{C}^M$.  In purely mathematical terms, 
the problem of phase retrieval is to recover a complex vector from its intensity measurements, namely 
the modulus of its inner product with these frame vectors. 
We show that any vector is uniquely determined (up to a global phase factor)
 from $4M-4$ generic measurements. To prove this, we identify the set of frames
 defining non-injective measurements with the projection of a real variety and bound its dimension. 

 \end{abstract}
\maketitle

\section{Introduction} 
In signal processing, a signal $x\in \C^M$ often cannot be measured
directly. Instead, one can only measure the absolute values of its
inner product with a fixed set of vectors $\Phi~=~\{\phi_1, \ldots,
\phi_N\} \in \C^M$. Here we take $\C^M$ with the 
inner product $\langle x , y \rangle = \sum_{m=1}^Mx_m \overline{y_m}$. 

An $N$-element complex \emph{frame} $\Phi$ is a collection of vectors $\phi_1, \ldots  , \phi_N$ which span $\C^M$. 
A complex frame $\Phi = \{\phi_n\}_{n=1}^N \subset \C^M$
defines $N$ \emph{intensity measurements} of a vector $x\in \C^M$, 
\begin{equation} \label{eq:intensity}
|\langle \phi_n, x\rangle|^2 \;\; = \;\; \phi_n^* x x^* \phi_n \;\;\;
\text{ for } \;\; n=1, \hdots, N,
\end{equation}
where we use $v^*$ to denote the conjugate transpose of a vector (or matrix) $v$.

The problem of \textit{phase retrieval} is to 
reconstruct a vector $x\in \C^M$ from its intensity measurements. Note that 
multiplying $x$ by a scalar of unit modulus does not 
change the measurements \eqref{eq:intensity}, so we can only reconstruct $x$
up to a global phase factor.  
For phase retrieval to be possible, any two vectors 
$x$ and $y$ with the same intensity measurements must differ by a scalar multiple 
of norm one,  namely $x = e^{\ii \theta} y$.
In other words, the non-linear map
\begin{equation}\label{eq:APhi}
\sA_\Phi \colon (\C^M/S^1)  \rightarrow  (\R_{\geq 0})^N \;\;
\text{ given by  }\;\; x \mapsto \left(| \langle x, \phi_n \rangle |^2 \right)_{n=1, \hdots, N}
\end{equation}
is injective, where $(\C^M/S^1)$ is obtained by identifying $x\in \C^M$ with $e^{\ii \theta} x$ for every $\theta \in [0,2\pi] $. 

Our main result states that $4M-4$ generic intensity measurements suffice to determine a
vector in $\C^M$. This proves part \!(b) of the ``$4M-4$ Conjecture'' made in \cite{BCMN:13}.

\begin{Thm}\label{thm:main}
If $N \geq 4M-4$, then for a generic frame $\Phi $ the map $\sA_\Phi$ is injective.
\end{Thm}

By generic we mean that $\Phi$ corresponds to a point in a non-empty
Zariski open subset of $\C^{M\times N} \cong (\R^{M\times N})^2$ (see Section~\ref{sec:algebra}). In particular, this theorem  
implies that when $N\geq 4M-4$,  there is an open dense set of frames $\Phi$ (in the Euclidean topology on 
$\C^{M \times N}$) for which $\sA_{\Phi}$ is injective. 
Part \!(a) of the conjecture in \cite{BCMN:13} says that this result is tight, i.e. that 
for $N < 4M-4$ the map $\sA_{\Phi}$ is never injective. This part is still open.

The history of this problem in the context of finite frames will be discussed in Section~\ref{sec:background}.
There, we also define some necessary concepts from algebraic geometry, such as generic points and the dimension of algebraic sets. 
In Section~\ref{sec:proof} (specifically on page~\pageref{proof:main}) we prove Theorem~\ref{thm:main}. 
A polynomial vanishing on the set of frames giving non-injective measurements is found and discussed in Section~\ref{sec.definingeqtns}.
Finally, in Section~\ref{sec:4M-5} we discuss what our algebraic approach can say about injectivity 
with fewer measurements.  We end by rephrasing the open part of conjecture of \cite{BCMN:13} 
in the language of real algebraic geometry and operator theory.

\section{Background}\label{sec:background}
Here we give a short review of the history of phase retrieval in the context of finite frames and review some needed terminology from algebraic geometry.

\subsection{The phase retrieval problem}
Phase retrieval is an old problem in signal processing, and the
literature on this subject is vast. However, in the context of finite
frame theory it was first considered 
Balan, Casazza, and Edidin \cite{BCE:06}.
 In \cite[Theorem~3.3]{BCE:06}, the authors show 
that the map $\sA_{\Phi}$ \eqref{eq:APhi} is injective for a generic frame $\Phi$ when $N\geq 4M-2$.
However, Bodmann and Hammen exhibit an explicit family of frames with
$4M-4$ vectors for which injectivity holds, which suggests the possibility of a better bound \cite{BoHa:13}. 
On the other hand  Heinosaari, Mazzarella and 
Wolf \cite{HMW:13}  used embedding theorems in homotopy theory
to show that $N\geq(4 + o(1))M$ is necessary for the injectivity of $\sA_{\Phi}$. 
Recently, Bandeira, Cahill, Mixon, and Nelson  \cite{BCMN:13} conjectured the following.

\begin{Conjecture4M-4} 
Consider a frame $\Phi = \{\phi_n\}_{n=1}^N \subseteq \C^M$ and the mapping 
$\sA_{\Phi} : (\C^M/S^1) \rightarrow (\R_{\geq 0})^N$ taking a vector $x$ to its intensity measurements 
$ \left(| \langle x, \phi_n \rangle |^2 \right)_{n=1, \hdots, N}$.
If $M\geq 2$ then the following hold.
\begin{itemize}
\item[(a)] If $N < 4M-4$, then $\sA_{\Phi}$ is not injective. 
\item[(b)] If $N\geq 4M-4$, then $\sA_{\Phi}$ is injective for generic $\Phi$.
\end{itemize}
\end{Conjecture4M-4}

In \cite{BCMN:13}, this conjecture was proved for $M=2,3$. Our 
Theorem~\ref{thm:main} establishes part \!(b). \medskip

Injectivity of the map $\sA_{\Phi}$ implies that phase retrieval is
  possible, but the problem of effectively reconstructing a vector
  from its intensity measurements is quite difficult. There have been many papers devoted to
  determining efficient reconstruction algorithms. For references 
we direct the reader to \cite{BCMN:13}.

\begin{Remark}
In \cite{BCE:06}, Balan, Casazza, and Edidin characterized frames giving injective measurements in the \textit{real} case. 
Precisely,
\cite[Theorem 2.8]{BCE:06} says that 
a real frame $\Phi$ defines injective measurements (on $\R^M/\{\pm1\}$)
if and only if $\Phi$ satisfies the {\em finite complement property}, 
which means that for every subset ${\mathbf S} \subset \{1, \ldots , N\}$ either 
${\{\phi_n\}}_{n \in {\mathbf S}}$ or its complement ${\{\phi_n\}}_{n \in {\mathbf S}^c}$ spans $\R^M$. 
In particular, if $N < 2M-1$ then the corresponding map $\sA_{\Phi}$ cannot be injective,  and if $N \geq 2M-1$ then for a  
generic frame $\Phi$, $\sA_{\Phi}$ is injective.

It would be very interesting to have an analogous characterization for complex frames. 
As a first step in this direction, in Section~\ref{sec.definingeqtns} we describe some polynomials that vanish on the set of 
 frames $\Phi$ for which $\sA_{\Phi}$ is non-injective.
\end{Remark}

\begin{Remark}
A frame $\Phi$ determines an $M$-dimensional subspace of $\C^N$ by taking 
the row span of the $M\times N$ matrix   whose columns are the vectors $\phi_n$, 
$1\leq n \leq N$. 
It was observed in \cite[Proposition 2.1]{BCE:06}, 
that if  
$\Phi$ and $\Phi'$ determine the same subspace in $\C^N$, then 
 $\sA_{\Phi}$ is injective if and only if $\sA_{\Phi'}$ is injective. In other 
 words, injectivity of $\sA_{\Phi}$ only 
 depends on subspace determined by $\Phi$.
 This subspace corresponds to a point in the Grassmannian $G(M,N)$ of $M$-dimensional subspaces of $\C^N$. 
 Thus there is a subset of $G(M,N)$ parameterizing frames for which $\sA_{\Phi}$ is injective. This approach
was applied in \cite{BCE:06}. 
\end{Remark} \smallskip

\subsection{Terminology from algebraic geometry}
\label{sec:algebra} 

Let $\mathbb{F}$ be a field (specifically $\mathbb{F} = \R$ or $\mathbb{F}=\C$). 
A subset of $\mathbb{F}^d$ defined by the vanishing of finitely many polynomials in $\mathbb{F}[x_1, \hdots, x_d]$ is called an {\em affine variety}. 
If these polynomials are homogeneous, then their vanishing defines a subset of projective space $\P(\mathbb{F}^{d})$, which is called a {\em projective variety}. 

The {\em Zariski topology} on $\mathbb{F}^d$ (or $\P(\mathbb{F}^d)$) is defined by declaring affine (resp. projective) varieties to be closed subsets. 
Note that a Zariski closed set is also closed in the Euclidean topology.
The complement of a variety is a {\em Zariski open} set.
A non-empty Zariski open set is open and dense in the
Euclidean topology. We say that a \emph{generic point} of $\mathbb{F}^d$ (or
$\P(\mathbb{F}^d)$) has a certain property if there is a non-empty Zariski open
set of points having this property.

The space of complex frames ${\mathcal F}(M,N)$ can be identified with $M\times N$ matrices of full rank, so it is a Zariski-open set in $\C^{M \times N}$. 
For the statement of  Theorem~\ref{thm:main} we identify $\C^{M \times N}$ with $(\R^{M \times N})^2$
and view ${\mathcal F}(M,N)$ as an open subset of $(\R^{M \times N})^2$.
Theorem~\ref{thm:main} then states that for $N\geq 4M-4$, there is a Zariski open subset $\mathcal{U}$ of ${\mathcal F}(M,N)$
 such that for every frame $\Phi$ corresponding to a point of $\mathcal{U}$, the map $\sA_{\Phi}$ is injective.  

In our main proof, we also rely heavily on the notion of the
dimension of a variety defined over $\C$.  For an
introduction and many equivalent definitions of the dimension of a
variety, see \cite[\S11]{Har:95} or \cite[Chapter 9]{CLO:07}.  In
particular, the \emph{dimension} of an 
irreducible variety (meaning that it is not the union of two proper subvarieties)
$X$ equals
the dimension of its tangent space at a generic point of
$X$.

We will also make use of the interplay between real and complex varieties.
Given a complex variety $X$ defined by equations with real coefficients 
we denote its set of real points by $X_{\R}$.

\section{Proof of Theorem \ref{thm:main}} \label{sec:proof}

We prove Theorem~\ref{thm:main} by showing that the subset of
$\C^{M\times N} \cong (\R^{M\times N})^2$ corresponding to frames 
$\Phi$ for which $\sA_{\Phi}$ is not injective 
is contained in a proper real algebraic
subset. The complement of this algebraic set is an open dense set
corresponding to frames $\Phi$ for which $\sA_{\Phi}$ is injective. 
A key ingredient of this proof is a reformulation, due to Bandeira, Cahill, Mixon, and Nelson \cite{BCMN:13}, 
of the injectivity of the map $\sA_\Phi \colon (\C^M/S^1) \to \R^N$ defined in \eqref{eq:APhi}.
\begin{Prop}[Lemma 9 \cite{BCMN:13}] \label{prop:dustin}
The map $\sA_\Phi$ is not injective if and only if there is a nonzero Hermitian matrix 
$Q \in \C^{M\times M}$ for which 
\begin{equation}\label{eq:rank2}
 \rank(Q) \leq 2 \;\;\text{ and }\;\;\phi_n^*Q\phi_n = 0 \;\; \text{ for each }1\leq n \leq N.
 \end{equation}
\end{Prop}
\medskip

We use this condition to translate injectivity of the map $\sA_{\Phi}$ into a question 
in algebraic geometry. Let $\C^{M\times M}_{\sym}$ denote the set of 
symmetric complex $M\times M$ matrices, and $\C^{M\times M}_{\Skew}$ the set of 
skew-symmetric complex $M\times M$ matrices.  

\begin{Def} Let $ {\B_{M,N}}$ denote the subset of 
$\P(\C^{M\times N}\times \C^{M\times N}) \times \P(\C^{M\times M}_{\sym} \times \C^{M\times M}_{\Skew})$  consisting of quadruples of matrices 
$([U, V], [X,Y])$ for which 
\begin{equation}\label{eq:BMNeq}
 \rank(X + \ii Y)\leq 2 \;\;\text{ and }\;\; 
u_n^T X u_n + v_n^T X v_n - 2u_n^T Yv_n=0 
\;\text{ for all }\; 1\leq n \leq N,  
 \end{equation}
where $u_n$ and $v_n$ are the $n$th columns of $U$ and $V$, respectively.

\end{Def}

The set $ {\B_{M,N}}$ is defined by the vanishing of polynomials in the entries of
$U$, $V$, $X$, and $Y$,  namely the $3\times 3$ minors of $X+\ii Y$ and the
polynomials 
$u_n^T X u_n + v_n^T X v_n - 2u_n^T Yv_n=0$. 
Note that these polynomials are homogeneous in the entries of $U,V$ and 
$X,Y$. In other words, they are invariant under scaling $U$ and $V$ by a non-zero scalar, and also $X$ and $Y$ by a non-zero scalar. 
Thus $ {\B_{M,N}}$ is a well-defined subvariety of the given product of projective spaces. 
Let $\pi_1$ be the projection onto the first coordinate,
\[\pi_1\; \colon \
\P\bigl(\C^{M\times N}\times \C^{M\times N}\bigl) \; \times \; \P\bigl(\C^{M\times M}_{\sym} \times \C^{M\times M}_{\Skew}\bigl) 
\;\;\;\to\;\;\; \P\bigl(\C^{M\times N}\times \C^{M\times N}\bigl).\]  
Recall that we use $X_{\R}$ to denote the set of real points of a complex variety $X$.

\begin{Prop}\label{prop:algcrit}
Let $\Phi = \{\phi_n\}_{n=1}^N  \subset \C^M$ be a complex frame. Write 
$\phi_n = u_n+ \ii v_n$ and let $U$ (resp. $V$) be the real matrix with 
columns $u_n$ (resp. $v_n$). Then the map $\sA_{\Phi}$ is injective 
if and only if $[U,V] $ does not belong to the projection $ \pi_1(( {\B_{M,N}})_{\R})$. 
\end{Prop}

\begin{proof}
Consider the incidence correspondence ${\mathcal I}$ of frames and Hermitian matrices given by 
\[
{\mathcal I} =  \left\{ (\Phi, Q) \in \C^{M\times N}  \times 
\C^{M\times M}_{\Herm} \; : \; Q\neq 0, \;
\rank(Q)\leq 2, \text{ and }\phi_n^*Q\phi_n = 0 \text{ for }n=1,\hdots, N  \right\}.
\]

Note that the conditions 
for $\cI$ involve complex conjugation, an inherently real operation. Thus 
we cannot view $\cI$ as a complex algebraic variety. However, complex conjugation is a polynomial on the real parts. 
So we decompose $\Phi$ and $Q$ into their real and imaginary parts,
i.e., 
$\Phi = U + \ii V$, 
$\phi_n = u_n + \ii v_n$
with $u_n,v_n \in \R^M$ and 
$Q = X + \ii Y$, with $X$ symmetric and $Y$ skew
symmetric. 
Then $\cI$ is linearly isomorphic over $\R$ to the subset $ \cJ$, 
\[\cJ =  \{(U,V,X,Y) \;:\;  X+\ii Y \neq 0, \; \rank (X + \ii Y)\leq 2, \text{ and }
u_n^T X u_n + v_n^T X v_n - 2u_n^T Yv_n=0 \},
\]
of the real vector space $\R^{M\times N} \times \R^{M\times N} \times \R^{M\times M}_{\sym} \times
\R^{M \times M}_{\Skew}$. 

By Proposition~\ref{prop:dustin}, $\sA_{\Phi}$ is injective if and only if 
$(U,V)$ is not contained in the projection of ${\mathcal J}$ onto the first two coordinates. 
Since $ (\B_{M,N})_{\R}$ is the projectivization 
of $\cJ$, $(U,V)$ is not contained in this projection if and only if 
$[U,V]\notin \pi_1((\B_{M,N})_{\R})$.
\end{proof}

To bound the dimension of the projection $\pi_1(\B_{M,N})$ we find the 
dimension of $\B_{M,N}$ itself. 

\begin{Thm}\label{thm:dimB}
The projective complex variety $\B_{M,N}$ has dimension $2MN -N+ 4M -6$. 
\end{Thm} 

\begin{proof}
Let $\B'_{M,N}$ be the subvariety of
$\P(\C^{M\times N}\times \C^{M\times N}) \times \P(\C^{M\times M})$  consisting of triples 
of matrices $([U, V], [Q])$ satisfying
\[ \rank(Q)\leq 2 \;\;\text{ and }\;\; (u_n-\ii v_n)^TQ(u_n+\ii v_n)=0 \;\text{ for all }\; 1\leq n \leq N,  \]
where $u_n$ and $v_n$ are the $n$th columns of $U$ and $V$, respectively.  
This is a  well defined subvariety of the product of projective spaces because the 
defining equations are homogeneous in each set of variables.  

Note that $ {\B_{M,N}}$ and $\B'_{M,N}$ are linearly isomorphic. 
We can identify $\C^{M\times M}_{\sym} \times \C^{M \times M}_{\Skew}$ with $\C^{M\times M}$ by the map $(X,Y) \mapsto X+\ii Y =Q$. 
Indeed any complex matrix $Q$ can be uniquely written as $Q = X + \ii Y$ where 
$X = (Q+Q^T)/2$ is a complex symmetric matrix and $Y = (Q-Q^T)/(2\ii)$ is a complex skew symmetric matrix. 
Hence it suffices to prove that $\B'_{M,N}$ has the desired dimension. 

We define $\pi_1$ and $\pi_2$ to be projections onto the first and second coordinates, namely
\[\pi_1\bigl([U,V],[Q]\bigl) \;=\;[U,V] \;\;\; \text{ and } \;\;\; \pi_2\bigl([U,V],[Q]\bigl) \;=\; [Q].\]   
We will determine the dimension of $\B'_{M,N}$ by finding the dimension of its second 
projection $\pi_2(\B'_{M,N})$ and the dimension of the preimages $\pi_2^{-1}(Q)$ 
for $Q \in \C^{M\times M}$. 

The image of $\B'_{M,N}$ under the projection $\pi_2$ is precisely the set of rank $\leq 2$ matrices in $\P(\C^{M\times M})$. 
To see that any rank $\leq 2$ matrix $Q$ belongs to this image, take any non-zero vector
$(u,v) \in \C^M \times \C^M$ satisfying the equation $(u - \ii v )^TQ (u + \ii v)^T=0$. 
(Such a vector exists because the zero set of this polynomial is a hypersurface in $\C^M \times \C^M$.) 
Now let $U$ and $V$ be the matrices with $N$ repeated columns $u_n = u$ and $v_n = v$.  Then $([U,V],[Q])$ belongs 
to $\B'_{M,N}$ and $[Q]$ is its image under $\pi_2$. 

The set of matrices of rank $\leq 2$ in $\C^{M\times M}$ is an irreducible (affine) variety of dimension $4M-4$ \cite[Prop. 12.2]{Har:95}.
So its projectivization in $\P(\C^{M\times M})$ has dimension $4M-5$, meaning
\[ \dim(\pi_2(\B'_{M,N})) \;\; = \;\ 4M-5.\]

Now fix $Q\in \pi_2(\B'_{M,N})$. We will show that the preimage, $\pi_2^{-1}(Q)$ in $\P(\C^{M\times N}\times \C^{M\times N})$ 
has dimension $2MN-N-1$.  By Lemma~\ref{lem:Qpoly} below, $Q$ defines a nonzero polynomial equation
\[ (u_n-\ii v_n)^TQ(u_n+\ii v_n) = 0\]
on the $n$-th columns of $U$ and $V$. For each pair of columns $(u_n,v_n)$, this polynomial defines a hypersurface 
of dimension $2M-1$ in $(\C^M)^2$. Thus the preimage of $Q$ in $\B'_{M, N}$ is a product of 
$N$ copies of this hypersurface in $((\C^M)^2)^N\cong (\C^{M\times N})^2$, one for each pair of columns $(u_n, v_n)$ for $1\leq n \leq N$. 
Therefore after projectivization, this preimage $\pi_2^{-1}(Q)$ has dimension $N(2M-1)-1 = 2MN-N-1$.
We put these together using the following theorem about dimensions of projections and their fibers \cite[Cor. 11.13]{Har:95}.
It states that the dimension of the projective variety $\B'_{M,N}$ is the sum of the dimension of the image of the projection 
$\pi_2(\B'_{M,N})$ and the minimum dimension of a preimage $\pi_2^{-1}(Q)$.  Since the dimension of the 
preimages is constant, we conclude that 
\[ \dim(\B'_{M,N}) \;\; =  \;\; \dim( \pi_2(\B'_{M, N})) + \dim(\pi_2^{-1}(Q)) \;\;=\;\; (4M-5) + (2MN-N -1). \vspace{-21pt} \]
\end{proof} \smallskip

Above we used that any non-zero matrix $Q$ imposes a nontrivial condition on each pair $(u,v)$ of 
columns  of $U$ and $V$. We now verify this statement.  

\begin{Lemma}\label{lem:Qpoly}
For a nonzero matrix $Q=(q_{\ell m}) \in \C^{M\times M}$, the polynomial
\[q(u,v) \;\;=\;\; (u-\ii v)^T Q (u+\ii v) \;\; \in  \;\; \C[u_1, \hdots, u_{M}, v_1, \hdots, v_M],\]
where $u = (u_1, \hdots, u_M)^T$ and $v = (v_1, \hdots, v_M)^T$, 
is not identically  zero.
\end{Lemma}

\begin{proof}
Computing explicitly the expression of $q(u,v)$, one has:  
$$
 q(u,v)=\sum_{1\leq m\leq M}q_{mm}(u_m^2+v_m^2)+ \sum_{1\leq \ell<m\leq M}(q_{\ell m}+q_{m \ell})(u_\ell u_m+v_\ell v_m)+\ii(q_{\ell m}-q_{m\ell})(u_\ell v_m-v_\ell u_m).
 $$
If the polynomial $q(u,v)$ is identically zero, then so are its coefficients, meaning
\begin{align*}
q_{mm}=0 &\;\;\mbox{ for all } 1\leq m\leq M, \\
q_{\ell m}+q_{m \ell}=0 &\;\;\mbox{ for all } 1\leq \ell <m \leq M, \text{ and }\\
q_{\ell m}-q_{m \ell}=0 &\;\; \mbox{ for all } 1\leq \ell <m \leq M.
\end{align*}
It follows that $Q$ is the zero-matrix. 
\end{proof}

By bounding the dimension of $\B_{M,N}$, we can bound the dimension of its
projection, which contains the frames $\Phi$ for which $\sA_{\Phi}$ is not injective, and thus prove our main theorem. 

\begin{proof}[Proof of Theorem~\ref{thm:main}] \label{proof:main}
By Proposition~\ref{prop:algcrit}, a pair of real $M\times N$ matrices $(U,V)$ for which $\sA_{U+\ii V}$ is not injective  
gives a point $[U,V]$ in $\pi_1((\B_{M,N})_{\R})\subset (\pi_1(\B_{M,N}))_{\R}$. 
The dimension of the 
projection is at most the dimension of the original variety  \cite[Cor. 11.13]{Har:95}. 
Thus the dimension of $\pi_1(\B_{M, N})$ can be bounded using Theorem~\ref{thm:dimB}:
\[ \dim\bigl(\pi_1(\B_{M, N})\bigl)  \;\; \leq \;\; \dim(\B_{M, N}) \;\; = \;\; 2MN + 4M-6 -N.\] 
When $N$ is $4M-4$ or higher, the dimension of this projection is \emph{strictly less} than $2MN-1$, 
which is the dimension of  $\P((\C^{M\times N})^2)$, the target of the projection $\pi_1$.  Thus the image of this projection 
is contained in a hypersurface defined by the vanishing of some polynomial.

This still holds when we restrict to real matrices $U$ and $V$.  In the real vector space $(\R^{M\times N})^2$, 
there is some nonzero polynomial that vanishes on all of the pairs $(U,V)$ for which $\sA_{U+\ii V}$ is not injective.
The complement of the zero-set of this polynomial is a Zariski open subset 
of  $(\R^{M\times N})^2$ and for any pair $(U, V)$ in this open set, $\sA_{U+\ii V}$ is injective.
\end{proof}

\section{A hypersurface containing bad frames}  \label{sec.definingeqtns} 
When $N\geq 4M-4$, the proof of our main theorem guarantees a polynomial that is zero on the set of 
frames $\Phi$ for which $\sA_{\Phi}$ is non-injective. Here we discuss how to obtain such a polynomial 
and compute its degree. 

Specifically, here we describe a polynomial in the variables $u_{mn}, v_{mn}$ for $1\leq m \leq M$ and 
$1\leq n \leq N$ vanishing on the projection $\pi_1(\B_{M,N})$. 
The projection from a product of projective spaces onto one of its coordinates, $\P^{d} \times \P^{e} \to \P^{d}$,
 is a closed map in the Zariski topology \cite[Theorem I.5.3]{Shaf:94}. Thus $\pi_1(\B_{M,N})$ is indeed
a subvariety of $\P((\C^{M \times N})^2)$, i.e. a closed set in the Zariski topology. The equations defining this projection 
can be in principle computed using symbolic computations involving eliminations, 
saturations and resultants. See for instance Chapter 3 in \cite{CLO:05} and the more advanced Chapter 12 and 13 in \cite{GKZ:94}.  
In particular we use \emph{resultants} (see \cite[Ch. 3, Thm. 2.3]{CLO:05}) which can be expressed by combining various determinants. 
The problem of expressing the resultant  in an efficient way, for example as a single determinant, is still a central topic in elimination theory, see for instance \cite{DaDi:01}.

\begin{Prop}
There is a nonzero polynomial in $\R[u_{11}, \hdots, u_{M (4M-4)}, v_{11}, \hdots, v_{M(4M-4)}]$ 
vanishing on the projection $\pi_1(\B_{M,4M-4})$ which has total degree $2\cdot (4M-4)\cdot 3^{(M-2)^2}$ 
and  has degree  $2\cdot 3^{(M-2)^2}$ in the set of column variables $\{u_{mn}, v_{mn}, \;m=1, \hdots, M\}$ for each $n$.
\end{Prop}

\begin{proof}
We compute this polynomial using resultants. Let $X$ and $Y$ be $M\times M$ symmetric and skew-symmetric 
matrices of variables $X=(x_{\ell m})$ and $Y=(y_{\ell m})$, and let $Z$ denote this collection of $M^2$ variables:
$$Z=\{x_{11},x_{12},\dots,x_{1M},x_{22}, \dots,x_{MM}, y_{12}, y_{13},\dots, y_{1M},y_{23}, \dots, y_{M-1M}\}.$$

Now set $E=(M-2)^2$. The $3\times 3$ minors of $X+\ii Y$ are cubic polynomials in the variables $Z$. 
Consider $E$ general linear combinations (with complex coefficients) of the $3\times 3$ minors, say $G_1,\dots,G_{E}$. 
To this set, add the $N=4M-4$ equations 
$$g_n\;=\;u_n^T X u_n + v_n^T X v_n - 2u_n^T Yv_n \; = 0 \;\;\;\; \text{ with }n=1,\dots,N$$ 
where $u_n$ and $v_n$ are the vector of variables $(u_{mn})_m$ and $(v_{mn})_m$.
In total we have $M^2$ polynomials in $M^2+2MN$ variables.  
To eliminate the variables $Z$ from our system of $M^2$ equations we take the resultant with respect to these variables
$$\Res(G_1,\dots,G_E,g_1,\dots,g_{N}).$$
This is a non-zero polynomial in the variables $u_{mn}$ and $v_{mn}$ that vanishes on $\pi_1(\B_{M,N})$. 
By \cite[Ch. 3, Theorem~3.1]{CLO:05}, such a resultant has total degree $2N3^E$
and it has degree $2\cdot 3^E$ in the entries of $u_n$ and $v_n$ for each $n=1,\dots,N$. 
\end{proof}
In practice, for small values of $M$ the collection of polynomials $G_1,\dots,G_{E}$ can be taken to be a subset of properly 
chosen $3\times 3$ minors (and not linear combinations of them). 
However for higher $M$ one needs to take linear combinations. 

When $N> 4M-4$, for every subset $S \subset \{1,\hdots, N\}$ of size $4M-4$, we can apply
the above construction to the corresponding columns of $U$ and $V$.   
The result is a nonzero polynomial vanishing on the set of bad frames and involving only the variables 
 $u_{mn}, v_{mn}$ where $n\in S$.

\begin{Example}[$M=2$, $N=4$]
Since the all matrices in $\C^{2\times 2}$ have rank $\leq 2$, the variety $\B_{2,4}$ is defined by the equations $g_n=0$ where
\[g_n \;\; = \;\;(u_{1n}^2+ v_{1n}^2 )x_{11} + 2(u_{1n} u_{2n} +  v_{1n} v_{2n} )x_{12} +(u_{2n}^2 + v_{2n}^2)x_{22}
+2(u_{2n} v_{1n} - u_{1n} v_{2n})y_{12}\]
for $n=1,\hdots, 4$.  These equations are linear in the variables $z_k\in Z= \{x_{11},x_{12},x_{22},y_{12}\}$. 
Thus for fixed $u_{mn},v_{mn}$, there is a nonzero solution to these equations if and only if 
the determinant of the Jacobian matrix 
\[\left(\frac{\partial g_n}{\partial z_k}\right)\!_{n,k} \;\; = \;\; 
\begin{pmatrix}
u_{11}^2+ v_{11}^2 &   2(u_{11} u_{21} +  v_{11} v_{21}) & u_{21}^2 + v_{21}^2& 2(u_{21} v_{11} - u_{11} v_{21})\\
u_{12}^2+ v_{12}^2 &   2(u_{12} u_{22} +  v_{12} v_{22}) &  u_{22}^2 + v_{22}^2& 2(u_{22} v_{12} - u_{12} v_{22})\\
u_{13}^2+ v_{13}^2 &  2(u_{13} u_{23} +  v_{13} v_{23}) & u_{23}^2 + v_{23}^2&   2(u_{23} v_{13} - u_{13} v_{23})\\
u_{14}^2+ v_{14}^2 &    2(u_{14} u_{24} +  v_{14} v_{24}) &u_{24}^2 + v_{24}^2&  2(u_{24} v_{14} - u_{14} v_{24})
\end{pmatrix} \]
is zero. This is the hypersurface defining $\pi_1(\B_{2,4})$, which has total degree $8$ and degree 2 in the entries of $u_n$ and $v_n$.  
If this determinant is \emph{non-zero}, then the map $\sA_{U+ \ii V}$ is injective. 
\end{Example} 

\begin{Example}[$M=3$, $N=8$]
For fixed $u_{mn}$, $v_{mn}$ the polynomials $g_n$ give $8$ linear equations
in the $9$ variables $Z=\{z_k\} =  \{ x_{11},  x_{12},  x_{13}, x_{22}, x_{23}, x_{33}, y_{12},  y_{13}, y_{23} \}$.
We can solve for this solution symbolically. To do this
consider Jacobian matrix: 
$$J=\left(\frac{\partial g_n}{\partial z_k}\right)\!_{n,k} \;\;\; \mbox{ with } \;\; 1\leq n\leq 8, \quad  1\leq k\leq 9.$$ 
The solution to the equations $g_1=\hdots g_8=0$ is then given by the $8\times 8$ sub-determinants 
$$z_k  \;\; = \;\;  D_k\;\; =\;\; (-1)^k \det( J^{\{k\}})$$ 
where $J^{\{k\}}$ is obtained by erasing the $k$-th column of $J$.  Note that $D_k$ 
has total degree $2\cdot 8$ and  degree $2$ the entries of $u_{n}$ and $v_{n}$ for each $n$. 
This solution gives a $3\times3$ matrix $X+\ii Y$ satisfying the desired equations $g_n=0$. 
In order for the pair $([U,V],[X,Y])$ to belong to $\B_{3,8}$, this matrix $X+\ii Y$ must have rank~$\leq 2$, meaning that its $3\times 3$ determinant, 
 \[ \det \left( \begin{array}{ccc} 
 D_1 &  D_2  + \ii   D_7  &   D_3  + \ii   D_8  \\ 
 D_2 - \ii   D_7  &   D_4  &    D_5  + \ii  D_9  \\
 D_3 - \ii  D_8  &   D_5  - \ii  D_9  &  D_6  \end{array}\right),\]
 must vanish. The vanishing of this determinant defines $\pi_1(\B_{3,8})$. As promised, it has total 
 degree  $2\cdot 8\cdot 3=48$ and degree $2\cdot 3 = 6$ in the entries of $u_n$ and $v_n$ for each $1\leq n\leq 8$. 
 \end{Example}

\begin{Remark}
The set of frames $\Phi$ such that $\sA_{\Phi}$ is not injective is $\pi_1((\B_{M,N})_\R)$. Since projective space is compact,  $\pi_1$ is a closed map with respect to the Euclidean topology. In particular, the locus of frames $\Phi$ for which $\sA_{\Phi}$ is non-injective is closed in the Euclidean topology on $\P((\R^{M\times N})^2)$. 
Note however, that the image of the set of real points of a variety need not be Zariski closed as the example below shows.
This means that there may be real points belonging to the projection $\pi_1(\B_{M,N})$ which are not the projection of real points of $\B_{M,N}$. 
That is, in principle there may be a real point $[U,V]$ in $\pi_1(\B_{M,N})$ 
whose corresponding frame $\Phi = U+\ii V$ is nonetheless injective. 

\begin{Example} \label{ex.realagishard}
Let $X \subset \C^2$ be the parabola defined by $x^2=y$
and let $\pi \colon X \to \C^1$ be the projection onto the second factor.
Since every real number has a complex square root, every point
in $\R$ is the image of a point of $X$. However, if $a < 0$ then
$a$ is not image of a real point of $X$. In particular the image of
$X_\R$ is the closed subset $\{a \geq 0\} \subset \R$. 
Any polynomial vanishing on $\pi(X_{\R})$ vanishes on all of $\R$, so the Zariski 
closure of $\pi(X_{\R})$ is all of $\R$. 
\end{Example}
\end{Remark}

\section{The case of fewer measurements} \label{sec:4M-5}

Here we use our algebraic reformulation to discuss some cases of part \!(a) of the $4M-4$ Conjecture.
We show that when $N\leq 4M-5$ the projection $\pi_1(\B_{M,N})$ fills the entire space and show that 
the projection of the real points $(\B_{M,N})_{\R}$ does this in the case $M =2^k+1$. 

\begin{Prop}\label{prop:partb}
For $N\leq 4M-5$, the projection $\pi_1(\B_{M, N})$ is all of $\P((\C^{M\times N})^2)$.
\end{Prop}
\begin{proof}
Fix $U$ and $V$ in $\C^{M\times N}$. Each pair of columns $u_n$ and $v_n$ define (at most) one linear 
condition on an $M\times M$ matrix $Q$, namely that $(u_n-\ii v_n)^TQ(u_n+\ii v_n)=0$. 
Thus in total $U$ and $V$ define (at most) $N$ linear conditions. The subvariety of $\P(\C^{M \times M})$ of matrices 
satisfying these linear conditions is a linear subspace 
\[ L_{\Phi} \;\; = \;\; \{Q\in \P(\C^{M\times M}) \;:\; (u_n-\ii v_n)^TQ(u_n+\ii v_n) = 0 \;\;\text{ for each }\;\;1\leq n\leq N\} \]
of dimension at least $M^2-1-N$.

On the other hand, 
the projective variety $H_2 \subset \P(\C^{M \times M})$ of matrices of rank $\leq 2$
has dimension $4M-5$ \cite[Prop. 12.2]{Har:95}.  When $N\leq 4M-5$, 
\[\dim L_\Phi \;+\; \dim H_2 \;\;\geq\;\;  M^2-1.\] 
Thus by \cite[Prop. 11.4]{Har:95}, there is a point in the intersection $L_\Phi \cap H_2$. 
This point corresponds to a matrix $Q$ of rank $\leq 2$ that satisfies the linear equations given by $U$ and $V$. 
We write $Q = X+ \ii Y$ where $X$ is a complex symmetric matrix and Y is complex skew symmetric.
Then $([U,V],[X,Y])$ belongs to $\B_{M, N}$ and $[U,V]$ is its image under the projection $\pi_1$. 
\end{proof}

Furthermore, when $N = 4M-5$, for generic matrices $(U,V) \in (\R^{M\times N})^2$, there should be finitely many 
matrices $Q \in \C^{M\times M}$ in the intersection $L_\Phi \cap H_2$ described above. 
Counting multiplicity, this number is given by the degree of $H_2$, namely 
\begin{equation}\label{eq:dM2}
d_{M,2} \;\; = \;\; \prod_{i=0}^{M-3} \frac{\binom{M+i}{2}}{\binom{2+i}{2}}
\end{equation}
(see for example \cite[Ex. 19.10]{Har:95}). Part \!(a) of the $4M-4$ Conjecture is equivalent to 
there always being a Hermitian matrix among these $d_{M,2}$ complex matrices. 

Both of the sets $L_{\Phi}$ and $H_2$ are invariant under the involution $Q \mapsto Q^*$, and so is
the finite set of matrices in their intersection.  In particular, when the degree $d_{M,2}$ is odd, 
this set must contain a fixed point, i.e. a Hermitian matrix.

\begin{Example}[$M=2$, \;$N= 3$]
\label{ex:M=2}
As shown in \cite{BCMN:13}, here part \!(a) of the $4M-4$ Conjecture holds, meaning that the intersection 
$L_{\Phi}\cap H_2$ contains a Hermitian matrix.  Every matrix has rank $\leq 2$, 
$H_2$ is all of $\C^{2\times 2}$, and $d_{2,2} =1$.  The projective linear space $L_{\Phi}$ is nonempty and invariant under 
the involution $Q \mapsto Q^*$. So it contains a Hermitian matrix. 
In this case, we recover the first part of the $4M-4$ conjecture 
from Proposition~\ref{prop:partb} and Proposition~\ref{prop:algcrit}.  
\end{Example}

\begin{Example}[$M=3$, \;$N= 7$]
\label{ex:M=3}
Here the variety of $\rank \leq 2$ matrices is defined by the $3\times 3$ determinant, meaning 
$d_{3,2}=3$.  Thus for generic $U,V \in \R^{3 \times 7}$, the intersection $L_{\Phi}\cap H_2$ contains three 
complex matrices. Since this intersection is invariant, 
at least one of these must be fixed under the involution  $Q \mapsto Q^*$. 
So in this case, we also recover the first part of the $4M-4$ conjecture 
from Proposition~\ref{prop:partb} and Proposition~\ref{prop:algcrit}.  
\end{Example}

More generally, the projection map 
$$\P((\C^{M \times N})^2) \; \times \; \P\left(\C^{M \times M}_{\sym} \times \C^{M \times M}_{\Skew}\right) \; \to \; \P((\C^{M \times N})^2)$$ 
maps real points to real points, so frames $\Phi$ for which $A_{\Phi}$ is not injective, namely $\pi_1((\B_{M,N})_{\R})$, are contained in the real points of the projection $(\pi_1(\B_{M,N}))_{\R}$. 
If we could show the reverse inclusion, then Proposition~\ref{prop:partb} and Proposition~\ref{prop:algcrit} would imply
the part \!(a) of the $4M-4$ conjecture. Unfortunately, as noted in Example~\ref{ex.realagishard} the image of the set of real points of variety 
need not equal the set of real points of the image.

Despite this subtlety, there is one case where we can use
algebro-geometric methods to prove part \!(a) of the $4M-4$ Conjecture.
\begin{Prop}\label{prop.noninj}
If $M = 2^k +1$ and $N \leq 4M-5$, then $\sA_{\Phi}$ is not injective. 
\end{Prop} 
\begin{proof} 
We will show that if $N \leq 4M -5$ then for every pair $[U,V] \in \P(\R^{M\times N}\times \R^{M\times N})$ 
there is some point $[X,Y] \in \P(\R^{M \times M}_{\sym} \times \R^{M \times M}_{\Skew})$ so that 
\[ \rank(X + \ii Y) \leq 2  \;\; \text{ and  }\;\; (u_n -\ii v_n)^T(X + \ii Y)(u_n +\ii v_n) =0 \;\; \text{for all $1 \leq n \leq N$.}\]
This would imply that the map $\sA_\Phi$ is not injective.

Fix any $[U,V]$ in $\P(\R^{M\times N}\times \R^{M\times N})$.
Following the notation from the proof of Proposition~\ref{prop:partb}, 
let $H_2$ denote the variety of matrices $[X,Y]\in \P(\C^{M \times M}_{\sym} \times \C^{M \times M}_{\Skew})$ for which $\rank(X+\ii Y) \leq 2$. Also, let $L_{\Phi}$ 
denote the projective linear space of pairs $[X,Y]$ satisfying the equations $(u_n -\ii v_n)^T(X+\ii Y)(u_n +\ii v_n) =0$. 

Here we use that $M = 2^{k} +1$. By Lemma \ref{lem.odddegree}
below, the subvariety of matrices of rank~$\leq 2$ has odd degree
in $\P(\C^{M \times M})$.  Hence its intersection with a linear subspace 
also has odd degree. We now use the fact that any projective variety 
defined over $\R$ and having odd degree has real point.
Thus the intersection $L_\Phi \cap H_2$ contains a real point. 
\end{proof}
\begin{Lemma} \label{lem.odddegree}
For $M = 2^k +1$, the variety of $M \times M$ matrices of rank~$\leq 2$ has odd degree.
\end{Lemma}
\begin{proof}
We use the classical formula for the degree \eqref{eq:dM2} of the variety of $M\times M$ matrices of 
rank at most two. 
Using Legendre's formula for the highest power of a prime dividing a factorial one can compute the highest power
of any prime dividing a binomial coefficient. Specifically, the highest
power of $p$ dividing $\binom{n+r}{r}$ is $s_p(n) + s_p(r) - s_p(n+r)$
where $s_p(n)$ is the sum of the digits in the base $p$ expansion of $n$.
Using this formula we see that the highest power of $2$ dividing $d_{M,2}$
is 
\begin{equation} \label{eq.1step}
 \left( \sum_{i=0}^{M-3} s_2(M+i) -s_2(M+i-2) \right) - \left(\sum_{i=0}^{M-3}
s_2(i+2) -s_2(i)\right).
\end{equation} 
Now if $M = 2^k+1$ and $0 \leq n \leq M-2$ then $s_2(M-1 + n)= s_2(n) +1$, so most of the terms in the two summations of \eqref{eq.1step} cancel leaving just four terms
\begin{equation} \label{eq.2step}
\left(s_2(M) -s_2(M-2)  \right) - \left(s_2(M-1) - s_2(M-3)\right).
\end{equation}
When $M = 2^{k} + 1$, $s_2(M)  =2$, $s_2(M-1) =1$, $s_2(M-2)=k-1$ and $s_2(M-3) = k-2$. Hence the expression in \eqref{eq.2step} is zero and $d_{M,2}$ is odd.
\end{proof}

Checking these formulas more closely reveals that $d_{M,2}$ is odd if and only if $M = 2^k+1$.  
Therefore a different approach would be necessary to prove the conjecture for other $M$.

\begin{Remark}
Proposition \ref{prop.noninj} is similar to, but does not seem to follow from, previous results \cite{D:08, HMW:13}. 
Heinosaari, Mazzarella and Wolf use embedding results from topology to show that when $N \leq 4M -2s_2(M-1) -4$, 
the map $\sA_{\Phi}$ is never injective \cite{HMW:13}.  In particular, if $M = 2^k +1$ then $s_2(M-1)=1$, and this bound gives 
 $N\leq 4M-6$, rather than $N\leq 4M-5$.
\end{Remark}

We end by rephrasing part~(a) of the $4M-4$ Conjecture. The first open case is $M=4$. 
\begin{Conjecture}\label{conj}
Let $\phi_1, \hdots, \phi_{4M-5}\in\C^M$ and consider the linear space $L_{\Phi}$ of 
$\C^{M\times M}_{\rm Herm}$, 
\[L_{\Phi} \;=\; \{Q \;:\; \phi_n^*Q\phi_n = 0 \;\text{ for }n=1, \hdots, 4M-5\}  
\; = \; {\rm span}\{\phi_1^{\:}\phi_1^*,\hdots, \phi_{4M-5}^{\:}\phi_{4M-5}^*\}^{\perp}.\]  
The $4M-4$ Conjecture states that $L_{\Phi}\subset \C^{M\times M}_{\rm Herm}$ always contains a matrix of rank $\leq 2$.
In other words, if we take $d=(M-2)^2+1$ Hermitian matrices $A_1, \hdots, A_{d}$
spanning $L_{\Phi}$, there is some linear combination
$x_1A_1+\hdots + x_{d}A_{d}$ with rank two.
\end{Conjecture}

\bigskip

{\bf Acknowledgements.} We are very grateful to the conference held at the AIM conference 
``Frame theory intersects geometry'' held in Summer 2013 and many of the people attending.  Among others, 
we would especially like to thank Bernhard Bodmann, Jameson Cahill, Matthew Fickus, Dustin Mixon, and Vlad Voroninski for helpful discussions. Aldo Conca would like to thank Carlos D'Andrea and Laurent Bus\'e for useful discussions concerning elimination and discriminants. Milena Hering was supported by NSF grant DMS-1001859.
Cynthia Vinzant was supported by an NSF postdoc DMS-1204447. 

\bigskip

\bibliography{refs}{}

\def\cprime{$'$}
\begin{thebibliography}{10}

\bibitem{BCE:06}
Radu Balan, Pete Casazza, and Dan Edidin.
\newblock On signal reconstruction without phase.
\newblock {\em Appl. Comput. Harmon. Anal.}, 20(3):345--356, 2006.

\bibitem{BCMN:13}
A.~Bandeira, J.~Cahill, D.~Mixon, and A.~Nelson.
\newblock Saving phase: Injectivity and stability for phase retrieval.
\newblock {\em arXiv:1302.4618}, 2013.

\bibitem{BoHa:13}
B.~Bodmann and N.~Hammen.
\newblock Stable phase retrieval with low-redundancy frames.
\newblock {\em arXiv:1302.5487}, 2013.

\bibitem{CLO:07}
David Cox, John Little, and Donal O'Shea.
\newblock {\em Ideals, varieties, and algorithms}.
\newblock Undergraduate Texts in Mathematics. Springer, New York, third
  edition, 2007.
\newblock An introduction to computational algebraic geometry and commutative
  algebra.

\bibitem{CLO:05}
David~A. Cox, John Little, and Donal O'Shea.
\newblock {\em Using algebraic geometry}, volume 185 of {\em Graduate Texts in
  Mathematics}.
\newblock Springer, New York, second edition, 2005.

\bibitem{DaDi:01}
Carlos D'Andrea and Alicia Dickenstein.
\newblock Explicit formulas for the multivariate resultant.
\newblock {\em J. Pure Appl. Algebra}, 164(1-2):59--86, 2001.
\newblock Effective methods in algebraic geometry (Bath, 2000).

\bibitem{D:08}
Donald~M. Davis.
\newblock Some new immersion results for complex projective space.
\newblock {\em Proc. Edinb. Math. Soc. (2)}, 51(1):45--56, 2008.

\bibitem{GKZ:94}
I.~M. Gel{\cprime}fand, M.~M. Kapranov, and A.~V. Zelevinsky.
\newblock {\em Discriminants, resultants, and multidimensional determinants}.
\newblock Mathematics: Theory \& Applications. Birkh\"auser Boston Inc.,
  Boston, MA, 1994.

\bibitem{Har:95}
Joe Harris.
\newblock {\em Algebraic geometry}, volume 133 of {\em Graduate Texts in
  Mathematics}.
\newblock Springer-Verlag, New York, 1992.
\newblock A first course, Corrected reprint of the 1992 original.

\bibitem{HMW:13}
Teiko Heinosaari, Luca Mazzarella, and Michael~M. Wolf.
\newblock Quantum tomography under prior information.
\newblock {\em Comm. Math. Phys.}, 318(2):355--374, 2013.

\bibitem{Shaf:94}
Igor~R. Shafarevich.
\newblock {\em Basic algebraic geometry. 1}.
\newblock Springer-Verlag, Berlin, second edition, 1994.
\newblock Varieties in projective space, Translated from the 1988 Russian
  edition and with notes by Miles Reid.

\end{thebibliography}
\bibliographystyle{plain}

\pagebreak

\footnotesize
\noindent Aldo Conca ({\tt conca@dima.unige.it})\\ 
Department of Mathematics,\\ 
University of Genova,\\  
Via Dodecaneso 35,\\
 I-16146 Genova, Italy\\  \smallskip

\noindent Dan Edidin ({\tt edidind@missouri.edu})\\
Department of Mathematics,\\ 
University of Missouri,\\ 
Columbia, Missouri 65211 USA\\  \smallskip

\noindent Milena Hering ({\tt m.hering@ed.ac.uk})\\
School of Mathematics and Maxwell Institute of Mathematics,\\ 
University of Edinburgh,\\ Edinburgh, EH9 3JZ, UK\\
 \smallskip

\noindent Cynthia Vinzant ({\tt vinzant@umich.edu})\\
Department of Mathematics,\\ University of Michigan,\\ Ann Arbor, MI 48109, USA

\end{document}